\documentclass[11pt,oneside]{amsart}

\usepackage{bm}

\usepackage{fix-cm}
\DeclareMathAlphabet{\mathcal}{OMS}{cmsy}{m}{n}

\usepackage{tikz-cd}

\usepackage{tcolorbox}

\usepackage[arrow,curve,matrix,tips,2cell]{xy}

\usepackage{fancyhdr}

\usepackage{anyfontsize}
\usepackage{mathrsfs}

\usepackage{accents}

\usepackage{bbm}

\usepackage{geometry}
\usepackage{amsmath}
\usepackage{amscd}
\usepackage{amssymb}
\usepackage{latexsym}
\usepackage{url}

\usepackage{enumitem}

\usepackage{graphicx}

\newtheorem{theorem}{Theorem}[section]
\newtheorem*{theorem*}{Theorem}
\newtheorem{lemma}[theorem]{Lemma}
\newtheorem*{lemma*}{Lemma}
\newtheorem{corollary}[theorem]{Corollary}
\newtheorem{proposition}[theorem]{Proposition}

\theoremstyle{definition}
\newtheorem{definition}[theorem]{Definition}
\theoremstyle{remark}
\newtheorem{remark}[theorem]{Remark}
\newtheorem*{definition*}{Definition}

\newtheorem{question}[theorem]{Question}
\newtheorem*{question*}{Question}

\theoremstyle{definition}
\newtheorem{example}[theorem]{Example}
\newtheorem{examples}[theorem]{Examples}

\makeatletter
\def\revddots{\mathinner{\mkern1mu\raise\p@
\vbox{\kern7\p@\hbox{.}}\mkern2mu
\raise4\p@\hbox{.}\mkern2mu\raise7\p@\hbox{.}\mkern1mu}}
\makeatother 
\newcommand{\bgl}{\begin{equation}} 
\newcommand{\egl}{\end{equation}}
\newcommand{\bgloz}{\begin{equation*}} 
\newcommand{\egloz}{\end{equation*}}
\newcommand{\bgln}{\begin{eqnarray}} 
\newcommand{\egln}{\end{eqnarray}}
\newcommand{\bglnoz}{\begin{eqnarray*}} 
\newcommand{\eglnoz}{\end{eqnarray*}}
\newcommand{\btheo}{\begin{theorem}}
\newcommand{\etheo}{\end{theorem}}
\newcommand{\btheooz}{\begin{theorem*}}
\newcommand{\etheooz}{\end{theorem*}}

\newcommand{\blemma}{\begin{lemma}}
\newcommand{\elemma}{\end{lemma}}
\newcommand{\blemmaoz}{\begin{lemma*}}
\newcommand{\elemmaoz}{\end{lemma*}}
\newcommand{\bproof}{\begin{proof}}
\newcommand{\eproof}{\end{proof}}
\newcommand{\bbew}{\begin{beweis}}
\newcommand{\ebew}{\end{beweis}}
\newcommand{\bremark}{\begin{remark}\em}
\newcommand{\eremark}{\end{remark}}
\newcommand{\bdefin}{\begin{definition}}
\newcommand{\edefin}{\end{definition}}
\newcommand{\bdefinoz}{\begin{definition*}}
\newcommand{\edefinoz}{\end{definition*}}
\newcommand{\bex}{\begin{example}}
\newcommand{\eex}{\end{example}}
\newcommand{\bexs}{\begin{examples}}
\newcommand{\eexs}{\end{examples}}

\newcommand{\bprop}{\begin{proposition}}
\newcommand{\eprop}{\end{proposition}}
\newcommand{\bcor}{\begin{corollary}}
\newcommand{\ecor}{\end{corollary}}
\newcommand{\bfa}{\begin{cases}} 
\newcommand{\efa}{\end{cases}}

\newcommand{\bquestion}{\begin{question}}
\newcommand{\equestion}{\end{question}}
\newcommand{\bquestionoz}{\begin{question*}}
\newcommand{\equestionoz}{\end{question*}}
\def\Nz{\mathbb{N}}

\def\1z{\mathbb{1}}
\newcommand{\onto}{\twoheadrightarrow} 

\def\SEMI{\mbox{$\times\kern-2pt\vrule height5pt width.6pt \kern3pt $}}

\newcommand{\defeq}{\mathrel{:=}} 

\newcommand{\dop}{\text{: }} 

\newcommand{\lge}{\left\{} 
\newcommand{\rge}{\right\}} 
\newcommand{\gekl}[1]{\lge #1 \rge} 
\newcommand{\menge}[2]{\gekl{ #1 \dop #2 }} 
\newcommand{\dom}{{\rm dom\,}}

\newcommand{\oset}[2]{%
  \mathop{#2}\limits^{\vbox to -1.66ex{%
  \kern -1.4ex\hbox{$#1$}\vss}}}

\newcommand{\pari}{\setlength{\parindent}{0.5cm} \setlength{\parskip}{0cm}}
\newcommand{\nopar}{\setlength{\parindent}{0cm} \setlength{\parskip}{0cm}}
\newcommand{\baltimes}[2]{
	\mathbin{_{#1}\times_{#2}}%
}

\usepackage{newtxtext}
\usepackage{newtxmath}

\begin{document}

\title{AF-action groupoid models for diagonal AH-algebras }

\thispagestyle{fancy}

\author{Ali I. Raad}

\address [] {Ali I. Raad, Mathematics and Science Department, American University in Bulgaria, Office 307 BAC, ul. Svoboda Bacharova 8, 2700 Blagoevgrad, Bulgaria} \email{araad@aubg.edu}\

\author{Jonathan Taylor}
\address{Jonathan Taylor, Institut für Mathematik, Universität Potsdam, Campus Golm, Haus 9, Karl-Liebknecht-Str. 24--25, 14476, Germany}
\email{jonathan.taylor@uni-potsdam.de}

\subjclass[2020]{Primary 46L05; Secondary 22A22}

\begin{abstract}
We show that an inductive AH-system with diagonal connecting maps describes an action of the canonical AF-groupoid on the spectrum of the canonical C\(^*\)-diagonal, and that the canonical groupoid model is given by the transformation groupoid associated to this action.
This divides the groupoid structure into two distinct aspects: the acting AF-groupoid (which is well-studied) and the unit space on which it acts.
We describe the unit space by enriching the Bratteli diagram with extra topological information, and apply these descriptions to examples of interest, including Villadsen algebras of the first kind.
\end{abstract}

\maketitle


\setlength{\parindent}{0cm} \setlength{\parskip}{0.5cm}

\section{Introduction}

Approximately homogeneous C$^*$-algebras (AH-algebras) arise as inductive limits of homogeneous C$^*$-algebras and have provided many examples important to the Elliott classification program.
A particular subclass of these are the diagonal AH-algebras: inductive limits built from matrix algebras over function spaces with diagonal connecting maps. 
These constraints on the inductive systems allow for finer analysis of the algebras they describe while still encompassing important examples. 
Elliott, Ho, and Toms \cite{EHT09} were able to characterise when such algebras are simple in terms of the inductive system, and further showed that any simple diagonal AH-algebra has stable rank one.
Recent work of Seth \cite[Theorem~3.7]{Seth26} characterises tensorial $K$-stability of these algebras in terms of the growth of the matrix dimensions of homogeneous algebras in the inductive system.

The building block algebras in a diagonal AH-system admit canonical groupoid models, and the connecting maps satisfy the criteria set out by Barlak and Li (see \cite[Theorem~3.6]{BL17} and \cite[Proposition~5.4]{Li18}) to ensure that the inductive limit C$^*$-algebra also admits an étale groupoid model. 
Since each groupoid in the inductive system is effective (so that the resulting C$^*$-algebra has a Cartan subalgebra) the groupoid model for the AH-algebra yields the Weyl groupoid corresponding to the inductive limit Cartan pair. 
In particular, the inductive limit of Cartan subalgebras yields a Cartan subalgebra of the inductive limit algebra.

Diagonal AH-systems give rise to a natural AF-system by considering only the matrix algebras over $\mathbb{C}$ in each of the elementary building blocks. 
The corresponding groupoid modelling the finite-dimensional subalgebras of each building block then acts on the diagonal subalgebra of the AH-algebra via conjugation by matrix units, and this passes to an action of a finite elementary AF-groupoid on a topological space. 
The diagonal conditions on the inductive system ensure that these actions are respected by the connecting maps, yielding an action of the canonical groupoid model of an AF-algebra on the spectrum of the  canonical inductive limit Cartan subalgebra of the AH-algebra. 
Our first main result shows that this action is enough to recover the Weyl groupoid of the AH-algebra as the transformation groupoid of this action.

\textbf{Theorem A.} (Theorem~\ref{thm: Af action on spectrum})
    Let \(\mathcal{G}\) be the groupoid model arising from the AF-subsystem of a diagonal AH-algebra $A$. 
    Let \(C\) be the canonical C$^*$-diagonal in \(A\). 
    There is an action of \(\mathcal G\) on \(\widehat{C}\) such that $A\cong C^*_{\mathrm r}(\mathcal{G}\ltimes \widehat C)$, where $\widehat{C}$ is the spectrum of $C$ and $\mathcal{G}\ltimes \widehat C$ is the transformation groupoid with respect to the action \(\mathcal G\curvearrowright \widehat C\).

Theorem~A divides the Weyl groupoid of a diagonal AH-algebra into two distinct components: the AF-groupoid associated to a canonical AF-subalgebra, and the action on the spectrum of the Cartan subalgebra.

We continue by describing the Gelfand spectrum of the canonical Cartan subalgebra of a diagonal AH-algebra. In the AF-setting, the spectrum is given by the space of infinite paths of the associated Bratteli diagram which start (or end, up to convention) at the root of the tree. 
To encode the additional topological information afforded by a diagonal AH-system, we employ labellings of Bratteli diagrams as developed by the first named author in \cite{Raa23}. 
This leads to the notion of labelled paths which carry a natural topology. 
We show that the spectrum of the canonical Cartan is given by the space of labelled paths in the Bratteli diagram, and that the resulting groupoid model for the AH-algebra is given by the groupoid of labelled tail-equivalences.

\textbf{Theorem B.} (Lemma \ref{lem:GBlambda}, Theorem \ref{thm:groupoid_as_labeled_Brat})
   Let $\mathcal{B}_{\lambda}$ be the labelled Bratteli diagram associated to a diagonal AH-algebra $A$. 
   The labelled tail-equivalence groupoid $\mathcal{G}_{\mathcal{B}_{\lambda}}$ of ${\mathcal{B}_{\lambda}}$ is isomorphic to $\mathcal{G}_{\mathcal B}\ltimes \widehat C$ and hence is a groupoid model for $A$.

We conclude by applying our results to some examples of interest. The first class of examples are the Villadsen algebras of the first kind \cite{Vil98}. This class includes the Goodearl algebras \cite{Goo} as well as Toms' examples of non-classifiable C$^*$-algebras \cite{Toms}. 

\textbf{Example C.} (Proposition~\ref{prop: example VillFirst}) The groupoid model for a general Villadsen algebra of the first kind $A$ is given by $$\mathcal{G}_{\mathcal{B}}\ltimes E_{\mathcal{P}_{\mathcal{B}}^{\infty}}$$ where $E$ is fibred over $\mathcal{P}_{\mathcal{B}}^{\infty}$, the space of infinite paths on the associated Bratteli diagram $\mathcal{B}$ of the AF-subsystem. If $X$ is the underlying compact Hausdorff space in the inductive system defining \(A\), then each fibre of \(E\) is of the form $\{*\}, X^{n}$, or $X^{\infty}$ (depending on the Villadsen algebra in question).

\textbf{Example D.} (Proposition~\ref{prop: exampleAHDynam}) Let \(X\) be a compact Hausdorff space and let \(\sigma\colon X\to X\) be a homeomorphism.
Let \(A\) be the AH-algebra model for the dynamical system \((X,\sigma)\), that is, the inductive limit of the system
\[\cdots\to C(X,M_{2^n})\xrightarrow{\phi_n}C(X,M_{2^{n+1}})\to\cdots\]
with connecting maps \(\phi_n\colon f\mapsto\begin{pmatrix}
    f&0\\
    0&f\circ\sigma
\end{pmatrix}\). Let \(G_{2^\infty}\) be the canonical groupoid model for the CAR algebra \(M_{2^\infty}\).
Then there is an action \(G_{2^\infty}\curvearrowright \{0,1\}^\Nz\times X\) such that \(G_{2^\infty}\ltimes (\{0,1\}^\Nz\times X)\) is the groupoid model for \(A\) arising from this AH-system.

The article is structured as follows.
Section~2 consists of preliminaries; we describe diagonal AH-systems and Cartan subalgebras, and in particular the canonical Cartan subalgebras arising in diagonal AH-algebras.
In Section~3 we describe the canonical AF-subsystem of a diagonal AH-system and the action of its groupoid model on the canonical Cartan subalgebra of the AH-algebra. 
We prove Theorem~A, describing the Weyl groupoid of the AH-algebra as the transformation groupoid for this action.
In Section~4 we describe the Gelfand spectrum of the canonical Cartan subalgebra of a diagonal AH-algebra as a space of labelled paths in the Bratteli diagram. 
Combining this with a natural notion of tail-equivalence for labelled paths, we prove Theorem~B, yielding another description of the Weyl groupoid of a diagonal AH-algebra which arises naturally from the analysis of path spaces on Bratteli diagrams.
Finally, in Section~5 we apply our results to some classical examples of diagonal AH-algebras.

\textbf{Acknowledgements.}

The authors thank Sven Raum for helpful comments on earlier drafts.

\section{Diagonal AH-systems and algebras}

We recall the definition of diagonal AH-systems as in \cite{EHT09}.
Let \(X\) and \(Y\) be compact Hausdorff spaces and let \(n,k\in\Nz\) be natural numbers.
An \emph{elementary diagonal map} is a \({}^*\)-homomorphism \(M_k\otimes C(X)\to M_{kn}\otimes C(Y)\) of the form
\[f\mapsto \mathrm{diag}(f\circ\lambda_1,\dots,f\circ\lambda_n)\]
for continuous functions \(\lambda_1,\dots,\lambda_n\colon Y\to X\), called the \emph{eigenfunctions}.
Here we identify \(M_k\otimes C(X)\) with \(C(X,M_k)\) and $M_k(C(X))$ in the usual way.

A \({}^*\)-homomorphism \(\varphi\colon\bigoplus_{i=1}^N M_{k_i}\otimes C(X_i)\to\bigoplus_{j=1}^M M_{\ell_j}\otimes C(Y_j)\) is a \emph{diagonal map} if each of the component \({}^*\)-homomorphisms \(\varphi_{j,i}\colon M_{k_i}\otimes C(X_i) \rightarrow M_{l_j}\otimes C(Y_j) \) is elementary diagonal (up to identification with the non-zero corner).
For the remainder of this article we shall consider unital diagonal maps. 

A \emph{diagonal AH-system} is an inductive system \((A_n,\varphi_n)_{n\in\Nz}\) of C\(^*\)-algebras \(A_n\) of the form
\begin{equation}\label{def:AH-form}
    A_n=\bigoplus\limits_{i=1}^{N_n}A_{n,i}, \; \; A_{n,i}=C(X_{n,i})\otimes M_{k_{n,i}}
\end{equation}
where each \(\varphi_n\colon A_n\to A_{n+1}\) is a unital diagonal map. We may as well assume that the maps are injective since this follows by \cite[Theorem~2.1,~Remark~2.2]{EGL05} and \cite[Theorem~2.2.10]{Li97}. 
A diagonal AH-system describes a \emph{diagonal AH-algebra} \(A:=\varinjlim_n (A_n,\varphi_n)_{n\in\Nz}\).

The class of diagonal AH-algebras is rich in examples, including AF-algebras \cite[Lemma~III.2.1]{dav96}, AI-algebras \cite[Section~3]{klaus2}, AH-algebra models for dynamical systems \cite[Example~2.5]{Niu}, Villadsen algebras of the first kind, particularly \cite{Vil98}, Goodearl algebras \cite{Goo} (see also \cite[Example~3.1.7]{Ror}), and Toms' examples of non-classifiable C*-algebras \cite{Toms}, \cite{Vil99}.

Each building block algebra \(A_n\) has a C\(^*\)-diagonal given by
\begin{equation}\label{eq:canonCartInBlocks}
    C_n=\bigoplus_{i=1}^{N_n} C_{n,i},\qquad C_{n,i}=C(X_{n,i})\otimes D_{k_{n,i}},
\end{equation} where $D_{k_{n,i}}$ is the diagonal subalgebra of $M_{k_{n,i}}.$
It follows quickly from their definition that diagonal connecting maps \(\varphi_n\) will restrict to maps between the canonical diagonal subalgebras.
The resulting inductive limit \(C:=\varinjlim_n(C_n,\varphi_n|_{C_n})_{n\in\Nz}\) embeds naturally into \(A\).

\subsection{Cartan subalgebras in inductive limits}

\bdefin[\cite{Kum,Ren08}]
\label{def:CartanSubalgebra}
A C\(^*\)-subalgebra $C$ of a C$^*$-algebra $A$ is a \emph{Cartan subalgebra} if
\nopar

\begin{itemize}
\item $C$ is a maximal commutative subalgebra;
\item $C$ is regular, that is, the set of normalisers $N_A(C) \defeq \menge{n \in A}{n C n^* \subseteq C \ \text{and} \ n^* C n \subseteq C}$ generates $A$ as a C$^*$-algebra;
\item there exists a faithful conditional expectation $P: \: A \onto C$.
\end{itemize}

A pair $(A,C)$, where $C$ is a Cartan subalgebra of a C$^*$-algebra $A$, is called a Cartan pair.
\pari

A Cartan subalgebra $C\subseteq A$ is a \emph{C$^*$-diagonal} if $(A,C)$ has the unique extension property, that is, every pure state of $C$ admits a unique extension to a (necessarily pure) state of $A$.
\edefin

Cartan subalgebras are a central tool for providing existence of twisted groupoid models for C\(^*\)-algebras.
We refer the reader to \cite{Sim} for an introduction to \'etale groupoids and their C\(^*\)-algebras, and to \cite{Ren08} for the Weyl groupoid and twist.

\begin{theorem}[{\cite[Theorems 5.2, 5.9]{Ren08}}, {\cite[Theorem 1.2]{Raa}}]\label{thm: Renaults theorem}
    Let $(\mathcal{G},\Sigma)$ be a locally compact Hausdorff \'etale and effective twisted groupoid.
    Then $(C^*_{\mathrm r}(\mathcal{G},\Sigma),C_0(\mathcal{G}^0))$ is a Cartan pair. Conversely, for every Cartan pair $(A,C)$ there exists a locally compact Hausdorff \'etale and effective twisted groupoid $(\mathcal{G},\Sigma)$ and an isomorphism of Cartan pairs $(A,C)\cong (C^*_{\mathrm r}(\mathcal{G},\Sigma),C_0(\mathcal{G}^0)).$
    Moreover, the twist \((\mathcal G,\Sigma)\) is unique among twists over locally compact Hausdorff \'etale groupoids leading to such an isomorphism.
\end{theorem}

If \(A=\varinjlim_n A_n\) is a diagonal AH-algebra, the inductive limit \(C\) of diagonals is a C\(^*\)-diagonal in the sense of Definition~\ref{def:CartanSubalgebra}. This follows by \cite[Proposition~5.4]{Li18} and \cite[Theorem~3.6]{BL17}. 
In addition, the Weyl groupoid and twist associated to this inductive limit arise from the Weyl groupoids and twists of the inductive system. Since the Weyl twist associated to the building block inclusions $C_n\subseteq A_n$ is trivial it follows that so is the Weyl twist for the inclusion $C\subseteq A.$

\section{The AF-subsystem and its action}

Given a diagonal AH-system \((A_n,\varphi_n)_n\), the description \eqref{def:AH-form} shows that each building block \(A_n\) contains an obvious finite-dimensional subalgebra
\begin{equation}\label{def:AF-subsyst}
    B_n=\bigoplus\limits_{i=1}^{N_n}B_{n,i}, \; \; B_{n,i}=1\otimes M_{k_{n,i}}.
\end{equation}
Diagonal maps \(\varphi_n\) preserve these subalgebras, so we may restrict the system \((A_n,\varphi_n)_{n\in\Nz}\) to an AF-system \((B_n,\varphi_n|_{B_n})_{n\in\Nz}\).
The corresponding AF-algebra \(B=\varinjlim_n (B_n,\varphi_n|_{B_n})_{n\in\Nz}\) then embeds naturally into \(A\) extending the inclusions \(B_n\subseteq A_n\), and we identify \(B\) with its image in \(A\) under this embedding.
Each \(B_n\) also has a diagonal \(D_n\) corresponding to the diagonal matrices in the direct summands of \(B_n\), and the diagonal connecting maps preserve these as well.
The resulting inductive limit \(D\subseteq B\) is again a C\(^*\)-diagonal of \(B\), and is exactly the intersection \(C\cap B\), where \(C\subseteq A\) is the canonical C\(^*\)-diagonal of \(A\) prescribed by the inductive system.

The seminal work of Bratteli \cite{B72} allows AF-systems to be faithfully encoded as Bratteli diagrams.

A \emph{Bratteli diagram} (see \cite{B72} for details) is a directed graph \(\mathcal B=(V,E,r,s)\) where the vertex and edge sets \(V\) and \(E\) admit partitions
\[V=\bigsqcup_{n=0}^\infty V_n,\qquad E=\bigsqcup_{n=0}^\infty E_n\]
into finite sets \(V_n\) and \(E_n\) satisfying \(s(E_n)=V_n\), \(r(E_n)=V_{n+1}\), and \(V_0=\{v_0\}\).
The unique point \(v_0\in V_0\) is called the \emph{root} of the diagram.

Given a Bratteli diagram \(\mathcal B\) one associates an AF-system \((B_n,\phi_n)_{n\in\Nz}\) with unital diagonal connecting maps in the following way: for each vertex \(v\) in the diagram let \(n_v\) be the number of paths from the root \(v_0\) to \(v\) (note that this is necessarily positive). 
Define \(B_n:=\bigoplus_{v\in V_n} M_{n_v}\) and define \(\phi_n\colon B_n\to B_{n+1}\) as the homomorphism which maps each \(M_{n_v}\) into \(M_{n_w}\) for \(v\in V_n\) and \(w\in V_{n+1}\) with multiplicity given by the number of edges \(v\to w\) in the diagram \(\mathcal B\). 

Any AF-system with diagonal connecting maps gives rise to an associated Bratteli diagram by reverse-engineering the above construction (see e.g. \cite[Chapter~2]{Raeburn_GraphAlg} for details but note that we have adopted the opposite path convention in this article).
By \cite[Proposition~2.12]{Raeburn_GraphAlg}, if \(\mathcal B\) is a Bratteli diagram associated to an AF-system \((B_n,\phi_n)_n\) then the inductive limit \(B:=\varinjlim_n B_n\) is isomorphic to a full corner of the graph algebra associated to \(\mathcal B\).
It is shown in \cite{KumjianPaskRaeburnRenault_GraphsGrpdsCKAlgs} that this full corner admits a groupoid model which we shall now describe.

Given a Bratteli diagram \(\mathcal B=(V,E,r,s)\) let \(\mathcal{P}_{\mathcal B}^\infty\) be the space of infinite paths \(x_1 x_2\cdots\) (read left-to-right) starting at the root \(v_0\).
For a path \(x=x_1\cdots\), let \(x_{[n,m]}\) denote the segment \(x_n\cdots x_m\).
The set \(\mathcal{P}_{\mathcal B}^\infty\) carries a natural topology with a base of open sets indexed by initial segments: for a finite path \(p=e_1\cdots e_n\) starting at \(v_0\) we define
\[Z_p:=\{x\in \mathcal{P}^\infty_{\mathcal B}: x_{[1,n]}=p\},\]
that is, elements of \(Z_p\) are the infinite paths starting at \(v_0\) with initial segment given by \(p\).
Equipped with this topology the infinite path space becomes a totally disconnected compact Hausdorff space.

Two infinite paths \(x,y\in \mathcal P^\infty_{\mathcal B}\) are \emph{tail-equivalent} if there is \(m\in\Nz\) such that \(x_{[m,\infty)}=y_{[m,\infty)}\). 
One readily verifies that this is an equivalence relation, and we define \(\mathcal{G}_{\mathcal B}\) as the pair groupoid with respect to this relation.
Explicitly, elements of \(\mathcal{G}_{\mathcal B}\) are pairs \((x,y)\) of tail-equivalent infinite paths starting at the root \(v_0\) and the groupoid composition is given by
\[(x,y)(y,z)=(x,z).\]
The topology on \(\mathcal{G}_{\mathcal B}\) admits a base of open bisections of the form 
\[Z_{p,q}=\{(x,y)\in \mathcal{G}_{\mathcal B}: x_{[1,n]}=p, y_{[1,n]}=q, x_{[n+1,\infty)}=y_{[n+1,\infty)}\},\]
where \(p,q\) are paths of length \(n\) in \(\mathcal B\) and terminating at the same vertex in $V_n$.
In particular, \(\mathcal G_{\mathcal B}\) is \'etale.
Combining \cite[Proposition~2.12]{Raeburn_GraphAlg} and \cite[Theorem~3.1]{KumjianPaskRaeburnRenault_GraphsGrpdsCKAlgs} yields that the groupoid \(C^*\)-algebra \(C^*_{\mathrm r}(\mathcal{G}_{\mathcal B})\) is isomorphic to the AF-algebra arising from the inductive limit of the AF-system associated to \(\mathcal B\).
The basic open sets \(Z_{p,q}\subseteq \mathcal{G}_{\mathcal B}\) form compact open bisections of \(\mathcal{G}_{\mathcal B}\) and the indicator functions \(1_{Z_{p,q}}\in C^*_{\mathrm r}(\mathcal{G}_{\mathcal B})\), for paths \(p,q\) of length \(n\), correspond to matrix units in the matrix algebra \(B_n\).

Any matrix unit in \(B_n\) normalises the subalgebra \(C_n\subseteq A_n\), and since the connecting maps in a diagonal AH-system map matrix units to (amplified) matrix units, we see that images in \(B_{n+k}\) of matrix units in \(B_n\) also normalise the subalgebra \(C_{n+k}\subseteq A_{n+k}\).
This property then passes to the inductive limit; the image of a building block matrix unit in the limit \(A\) normalises \(C\).
By \cite[2.6]{Kum}, these matrix units implement partial homeomorphisms on the Gelfand dual \(\widehat C\) of \(C\).
These partial homeomorphisms implement an action of \(\mathcal G_{\mathcal B}\) on \(\widehat C\).

We recall the definition of a groupoid action on a topological space and the construction of the transformation groupoid.

For topological spaces \(X,Y,Z\) and continuous maps \(f\colon X\to Z\) and \(g\colon Y\to Z\) we write
\[X\baltimes{f}{g} Y:=\{(x,y)\in X\times Y: f(x)=g(y)\}\]
for the \emph{pullback} or \emph{fibre product} of \(X\) and \(Y\) along \(f\) and \(g\). We treat \(X\baltimes{f}{g} Y\) with the subspace topology of the product topology.

For an \'etale groupoid \(\mathcal{G}\) and topological space \(X\), an \emph{action} \(\mathcal{G}\curvearrowright X\) of \(\mathcal{G}\) on \(X\) consists of a continuous \emph{anchor map} \(\rho\colon X\to \mathcal{G}^{(0)}\) and a continuous map \(\cdot\colon \mathcal{G}\baltimes{s}{\rho}X\to X\) denoted \((\gamma,x)\mapsto\gamma\cdot x\) satisfying
\begin{enumerate}
    \item \(\rho(x)\cdot x=x\) for all \(x\in X\),
    \item \(\rho(\gamma\cdot x)=r(\gamma)\) for all \((\gamma,x)\in \mathcal{G}\baltimes{s}{\rho}X\), and
    \item \((\eta\gamma)\cdot x=\eta\cdot (\gamma\cdot x)\) for all \(\eta,\gamma\in G\) and \(x\in X\) with \(s(\eta)=r(\gamma)\) and \(s(\gamma)=\rho(x)\).
\end{enumerate}
Note that the second condition ensures that either side of the equation in the third condition is defined if the other is.

Given an action \(\mathcal{G}\curvearrowright X\) of an \'etale groupoid on a topological space with anchor map \(\rho\), the \emph{transformation groupoid} \(\mathcal{G}\ltimes X\) is defined in the following way.
As a topological space, \(\mathcal{G}\ltimes X\) is identical to the pullback \(\mathcal{G}\baltimes{s}{\rho}X\).
The range and source maps are given by \(r(\gamma,x)=(r(\gamma),\gamma\cdot x)\) and \(s(\gamma,x)=(s(\gamma),x)\) and the composition is given by
\[(\eta,\gamma\cdot x)(\gamma,x)=(\eta\gamma,x)\] for $(\eta,\gamma)\in\mathcal{G}^{(2)}.$
These operations make \(\mathcal{G}\ltimes X\) an \'etale groupoid.
The map \(X\to (\mathcal{G}\ltimes X)^{(0)}\) sending \(x\) to \((\rho(x),x)\) is then a homeomorphism, and we often identify the unit space \((\mathcal{G}\ltimes X)^{(0)}\) and \(X\) under this homeomorphism.

We can now describe the action of the AF-groupoid \(\mathcal G_{\mathcal B}\) on the spectrum \(\widehat C\).

\begin{proposition}\label{prop:AFGrpdActOnAHSpace}
    Let \(A\) be a diagonal AH-algebra with canonical diagonal \(C\), and let \(B\) be the associated AF-subalgebra.
    Let \(\widehat{C}\) be the Gelfand dual of the C$^*$-diagonal \(C\).
    Conjugation by matrix units in \(B_n\subseteq B\) defines an action of \(\mathcal{G}_{\mathcal B}\) on \(\widehat{C}\).
    Explicitly, the anchor map \(\rho\colon \widehat{C}\to \mathcal{P}^\infty_{\mathcal B}\) for the action is the Gelfand dual of the inclusion \(C\cap B\subseteq C\) and the action composition is given by
    \begin{equation}\label{eq:AFActionOnDiag}
        \left((x,y)\cdot \chi\right)(c)=\chi(1_{Z_{p,q}}^*c1_{Z_{p,q}}),
    \end{equation}
    for \((x,y)\in \mathcal{G}_{\mathcal B}\), \(\chi\in\widehat{C}\), and \(c\in C\) with \(y=\rho(\chi)\) and \(p,q\) finite paths in \(\mathcal B\) starting at \(v_0\) with \((x,y)\in Z_{p,q}\).
    \begin{proof}
        We first check that the conjugation formula \eqref{eq:AFActionOnDiag} is well-defined.
        Fix \((x,y)\in \mathcal{G}_{\mathcal B}\), \(\chi\in \widehat C\) with \(\rho(\chi)=y\), and \(c\in C\).
        Since \(1_{Z_{p,q}}\) is a matrix unit in \(B_n\) for any finite paths \(p\) and \(q\) of length \(n\) starting at \(v_0\) and terminating at the same vertex, the diagonal connecting maps ensure that the images of \(1_{Z_{p,q}}\) in later building blocks \(A_{n+k}\) normalise the diagonal \(C_{n+k}\) for all \(k\in\Nz\).
        Hence \(1_{Z_{p,q}}\) is a normaliser of \(C\) in \(A\) so the right-hand side of \eqref{eq:AFActionOnDiag} is coherent.
        We now show that the right-hand side of \eqref{eq:AFActionOnDiag} is independent of the choice of initial segments \(p\) and \(q\) of \(x\) and \(y\) respectively.
        Fix alternative initial segments \(p'\) and \(q'\) with \((x,y)\in Z_{p',q'}\).
        Without loss of generality we may assume \(p\) and \(q\) are shorter initial segments than \(p'\) and \(q'\). 
        We then have \(Z_{p'}\subseteq Z_p\), \(Z_{q'}\subseteq Z_q\), and hence 
        $$Z_{p'}Z_{p,q}=Z_{p',q'}=Z_{p,q}Z_{q'}.$$  
        As \(\rho(\chi)=y\) and \(y\in Z_{q'}\subseteq Z_q\) we have \(\chi(1_{Z_{q'}})=1\).
        We compute
        \begin{align*}
            \chi(1_{Z_{p',q'}}^*c1_{Z_{p',q'}})&=\chi(1_{Z_{q'}}1_{Z_{p,q}}^*c1_{Z_{p,q}}1_{Z_{q'}})\\
            &=\chi(1_{Z_{q'}})\chi(1_{Z_{p,q}}^*c1_{Z_{p,q}})\chi(1_{Z_{q'}})\\
            &=\chi(1^*_{Z_{p,q}}c1_{Z_{p,q}})
        \end{align*}        
        showing that the expression for \((x,y)\cdot\chi\) does not depend on choice of initial segments \(p\) and \(q\) with \((x,y)\in Z_{p,q}\).

        Under the identification \(\mathcal{P}^\infty_\mathcal B\ni y\mapsto (y,y)\in \mathcal{G}_{\mathcal B}^{(0)}\), verifying \(\rho(\chi)\cdot \chi=\chi\) is straightforward.
        We now show that \(\rho((x,y)\cdot\chi)=r(x,y)=x\).
        As \(\rho\) is the Gelfand dual map to the inclusion \(C(\mathcal{P}^\infty_\mathcal B)\subseteq C\), we must show that \((x,y)\cdot\chi(f)=f(x)\) for all \(f\in C(\mathcal{P}^\infty_{\mathcal B})\).
        Since \(\mathcal{P}^\infty_{\mathcal B}\) has a base of compact open sets of the form \(Z_t\) for initial path segments \(t\), it suffices to verify the equality on indicator functions on these sets.
        Fix such a finite path \(t\) starting at \(v_0\).
        The path \(t\) forms an initial segment of \(x\) if and only if \(1_{Z_t}(x)=1\).
        In this case, set \(p:=t\) and let \(q\) be the initial segment of \(y\) of the same length.
        Then \(1_{Z_{p,q}}^*1_{Z_t}1_{Z_{p,q}}=1_{Z_q}\) and so \((x,y)\cdot\chi(1_{Z_t})=\chi(1_{Z_q})=1\) since \(\rho(\chi)=y\) and \(y\in Z_q\).
        Conversely, if \(t\) does not form an initial segment of \(x\) then let \(p\) be the initial segment of \(x\) of the same length (or longer) than \(t\), and let \(q\) be the corresponding initial segment of \(y\).
        Then \(Z_tZ_{p,q}=\emptyset\) since \((w,z)\in Z_{p,q}\) implies \(p\) is an initial segment of \(w\), whereby \(t\) cannot possibly be an initial segment of \(w\).
        Hence \(1_{Z_t}1_{Z_{p,q}}=0\) and hence \((x,y)\cdot\chi(1_{Z_t})=\chi(0)=0=1_{Z_t}(x)\), as required.

        Finally, to see that \((z,x)\cdot((x,y)\cdot \chi)=(z,y)\cdot\chi\) for \((z,x)\in \mathcal{G}_\mathcal B\) we note that if \(p\), \(q\), and \(r\) are initial segments of \(x\), \(y\), and \(z\) respectively of the same length and terminating at the same vertex, then \(Z_{r,p}Z_{p,q}=Z_{r,q}\) is a bisection neighbourhood of \((z,x)(x,y)=(z,y)\).
        This implies \(1_{Z_{r,p}}1_{Z_{p,q}}=1_{Z_{r,q}}\) and so for any \(c\in C\) we then have
        \begin{align*}
            \left((z,x)\cdot\left((x,y)\cdot\chi\right)\right)(c)&=\left((x,y)\cdot\chi\right)(1_{Z_{r,p}}^*c1_{Z_{r,p}})\\
            &=\chi(1_{Z_{p,q}}^*1_{Z_{r,p}}^*c1_{Z_{r,p}}1_{Z_{p,q}})\\
            &=\chi(1_{Z_{r,q}}^*c1_{Z_{r,q}})\\
            &=\left((z,y)\cdot\chi\right)(c)
        \end{align*}
        as required.
    \end{proof}
\end{proposition}

As the action \(\mathcal G_{\mathcal B}\curvearrowright\widehat C\) is implemented by conjugation with normalisers of \(C\) in the AH-algebra \(A\), it is reasonable to expect a natural mapping from the transformation groupoid \(\mathcal G_{\mathcal B}\ltimes\widehat C\) into the Weyl groupoid \(\mathcal G_A\) of the inclusion \(C\subseteq A\). 
We show that this is indeed the case, and that this mapping is an isomorphism.
We recall briefly the construction of the Weyl groupoid from a Cartan inclusion \(C\subseteq A\).
Let \(X\) denote the Gelfand spectrum of \(C\).
For a normaliser \(n\in N_C(A)\) set \(\dom(n):=\{x\in X: n^*n(x)\neq 0\}\) and \(\mathrm{ran}(n)=\dom(n^*)\).
By \cite[2.6]{Kum}, \(n\) defines a partial homeomorphism \(\alpha_n\colon\dom(n)\to\mathrm{ran}(n)\) satisfying 
\begin{equation}\label{eq:unique_homeo}
n^*fn(x)=n^*n(x)f(x)
\end{equation}
for all \(x\in\dom(n)\).
The Weyl groupoid is then the groupoid of germs \(\mathcal G\) arising from this family of partial homeomorphisms.
Elements of $\mathcal{G}$ are equivalence classes $[\alpha_n,x]$ where $x\in\dom(n)$ subject to the \emph{germ relation}: two pairs $(\alpha_{n_1},x_1)$ and $(\alpha_{n_2},x_2)$ are equivalent if and only if $x_1=x_2$ and there exists an open neighbourhood $U\subset \dom(n_1)\cap\dom(n_2)$ of \(x\) such that $\alpha_{n_1}|_{U}=\alpha_{n_2}|_U$.

\begin{theorem}\label{thm: Af action on spectrum}
    Let \(\mathcal{G}_{\mathcal B}\ltimes \widehat C\) be the transformation groupoid with respect to the action \eqref{eq:AFActionOnDiag} and let \(\mathcal{G}_A\) be the Weyl groupoid of the Cartan inclusion \(C\subseteq A\).
    There is an isomorphism \(F\colon \mathcal{G}_{\mathcal B}\ltimes \widehat C\to \mathcal{G}_A\) of topological groupoids given by
    \[F((x,y),\chi)=[\alpha_{1_{Z_{p,q}}},\chi],\]
    where \(p\) and \(q\) are initial segments of \(x\) and \(y\) of the same length such that \((x,y)\in Z_{p,q}\).
    \begin{proof}
        We first show that $F$ is well-defined. Since $(x,y)\in Z_{p,q}$ we have $y\in Z_q$, and since \(\rho(\chi)=y\) we obtain that \(\chi\) lies in the domain of \(\alpha_{1_{Z_{p,q}}}\). If $(x,y)\in Z_{p^{\prime},q^{\prime}}$ we may without loss of generality assume $p^{\prime}$ and $q^{\prime}$ are extensions of $p$ and $q$ respectively, necessarily agreeing after the terminal node of $p$ (or $q$). It follows that $\chi$ is also in the domain of $\alpha_{1_{Z_{p^{\prime},q^{\prime}}}}$ and $\alpha_{1_{Z_{p,q}}}\vert_{\rho^{-1}(Z_{q^{\prime}})}=\alpha_{1_{Z_{p^{\prime},q^{\prime}}}}\vert_{\rho^{-1}(Z_{q^{\prime}})}$ as \(p^\prime\) extends \(p\). Hence $[\alpha_{1_{Z_{p,q}}},\chi]=[\alpha_{1_{Z_{p^{\prime},q^{\prime}}}},\chi]$, demonstrating that $F$ is well-defined.
        
        To see that \(F\) defines a groupoid homomorphism, we fix a composable pair \((((z,x),\tau), ((x,y),\chi))\) in \(\mathcal{G}_{\mathcal B}\ltimes \widehat C\). Note that this implies that $\tau=(x,y)\cdot \chi.$ Fix initial segments \(t,p,q\) of \(z,x,y\) respectively, such that \((z,x)\in Z_{t,p}\) and \((x,y)\in Z_{p,q}\). The formulae \eqref{eq:unique_homeo} and \eqref{eq:AFActionOnDiag} ensure that \(\alpha_{1_{Z_{p,q}}}(\chi)=(x,y)\cdot \chi=\tau\), ensuring that \(F((z,x),\tau)\) and \(F((x,y),\chi)\) are composable in \(\mathcal G_A\).  
        We then have
        \[F((z,x),\tau)F((x,y),\chi)=[\alpha_{1_{Z_{t,p}}},\tau][\alpha_{1_{Z_{p,q}}},\chi]=[\alpha_{1_{Z_{t,q}}},\chi]=F((z,y),\chi),\]
        whereby \(F\) is a groupoid homomorphism.

        To see that \(F\) is open we fix a basic open bisection \((Z_{p,q}\times U) \cap (\mathcal{G}_{\mathcal B}\ltimes \widehat C)\) for finite paths \(p,q\) in \(\mathcal B\) and \(U\subseteq \widehat C\).
        The image of this set under \(F\) is the set \(\{[\alpha_{1_{Z_{p,q}}},\chi]:\chi\in U\cap \mathrm{dom}(\alpha_{1_{Z_{p,q}}})\}\subseteq \mathcal{G}_A\) which is a basic open set in the Weyl groupoid, so \(F\) is open.

        We now show that \(F\) is a bijection. Note that if  \(F((x,y),\chi)\in \mathcal{G}_A^{(0)}\) it follows that $[\alpha_{1_{Z_{p,q}}},\chi]$ is a unit which means $p=q$ and so $x=y$.
        Hence \(F\) is injective.
        We now show that \(F\) is surjective.
        Combining \cite[Theorem~3.6]{BL17} and \cite[Lemma~6.7, Theorem~6.12]{Tay} there is an isomorphism \(\mathcal{G}_A\cong \overline{ G}:=\varinjlim_n\varprojlim_m \mathcal{G}_n\ltimes\widehat{C_{n+m}}\) realising \(\mathcal{G}_A\) as the inductive limit of projective limits of transformation groupoids \(\mathcal{G}_n\ltimes\widehat{C_{n+m}}\), where \(\mathcal{G}_n\) is the Weyl groupoid associated to the Cartan inclusion \(C_n\subseteq A_n\).
        Moreover, this isomorphism can be chosen to map \([\alpha_{1_{Z_{p,q}}},\chi]\) for initial segments \(p\) and \(q\) and character \(\chi\in \widehat{C}\) in the domain of $\alpha_{1_{Z_{p,q}}}$ to the class of the sequence 
        \begin{equation}\label{eq:indLimSeqForm}
            ([\alpha_{1_{Z_{p,q}}},\chi_{n}],\chi_{n+m})_m
        \end{equation} 
        where \([1_{Z_{p,q}},\chi_{n}]\in \mathcal{G}_n\ltimes \widehat{C_{n}}\) and \((\chi_{n+m})_m\) is a sequence in the projective limit \(\varprojlim_m \widehat{C_{n+m}}\cong \widehat{C}\) realising \(\chi\). 
        Since the characteristic functions \(1_{Z_{p,q}}\) form matrix units for the matrix algebras in \(A_n\), every element of the Weyl groupoid \(\mathcal{G}_n\) can be realised as such a class \([\alpha_{1_{Z_{p,q}}},\chi_n]\), so every element of the inductive limit groupoid \(\overline{G}\) is represented by a class of the form in \eqref{eq:indLimSeqForm}.
        Hence \(F\) composed with the isomorphism \(\mathcal{G}_A\cong \overline{G}\) is surjective, whereby \(F\) is necessarily surjective.

        Finally we show continuity.
        The Weyl groupoid \(\mathcal G_A\) has a basis of open sets of the form \(\{[\alpha_n,\chi]: \chi\in\dom(n)\cap U\}\) for some \(n\in N_C(A)\) and \(U\subseteq\widehat C\) open.
        Using surjectivity of \(F\) we may cover this set with sets of the form \(\{[\alpha_{1_{Z_{p,q}}},\chi]: \chi\in\dom(1_{Z_{p,q}})\cap V\}\) for some paths \(p,q\) in \(\mathcal B\) and \(V\subseteq \widehat C\) open.
        The preimages of these sets under \(F\) have the form \((Z_{p,q}\times V)\cap(\mathcal G_B\ltimes\widehat C)\), which is open in the transformation groupoid \(\mathcal G_B\ltimes\widehat C\).
        Thus \(F\) is continuous.
    \end{proof}
\end{theorem}

The following corollary can be obtained by the techniques in \cite{BL17}, but in our setting the proof simplifies and so we include it for completeness.

\begin{corollary}\label{cor:dense normalizers}
    The union of building block normalisers \(\bigcup_{n\in\Nz}N_{C_n}(A_n)\) is dense in \(N_C(A)\).
    \begin{proof}
        Using Theorem~\ref{thm: Af action on spectrum} it suffices to prove the result for the inclusion \(C:=C(\widehat C)\subseteq A:=C^*_{\mathrm{r}}(\mathcal{G}_{\mathcal B}\ltimes\widehat C)\). 
        By \cite[Lemma~5.8]{Ren08} it suffices to show that the normalisers arising from building blocks \(C_n\subseteq A_n\) are dense in the normalisers belonging to \(C_c(\mathcal{G}_{\mathcal B}\ltimes\widehat C)\), and \cite[Corollary~5.6]{Ren08} implies that a normaliser \(m\in N_C(A)\) with compact support is supported on an open bisection \(U\subseteq \mathcal{G}_{\mathcal B}\ltimes\widehat C\).
        By a partition of unity argument, it suffices to consider \(m\in C_c(\mathcal{G}_{\mathcal B}\ltimes\widehat C)\) with support contained in a basic bisection of the form \((Z_{p,q}\times V)\cap (\mathcal{G}_{\mathcal B}\ltimes\widehat C)\), where \(p,q\) are finite paths of the same length in the Bratteli diagram starting at the root \(v_0\) and terminating at the same vertex, and \(V\subseteq\widehat C\) is an open subset.
        Then \(f:=1_{Z_{p,q}}^*m\in C(\widehat C)=\varinjlim_nC_n\) and hence we can find a sequence \(f^j\in C_j\) converging to \(f\).
        For large enough $j$, \(1_{Z_{p,q}}f_j\) is a normaliser for the inclusion \(C_j\subseteq A_j\) and \(1_{Z_{p,q}}f_j\to 1_{Z_{p,q}}f=1_{Z_{p,q}}1_{Z_{p,q}}^*m\).
        But \(1_{Z_{p,q}}1_{Z_{p,q}}^*m=m\) since \(m\) has support contained in \(Z_{p,q}\times \widehat C\), completing the proof.
    \end{proof}
\end{corollary}

\section{The spectrum of the diagonal}

Theorem~\ref{thm: Af action on spectrum} divides the description of the Weyl groupoid for a diagonal AH-algebra \(A\) into two parts: the AF-component described by the Bratteli diagram and AF-groupoid, and the spatial component given by the Gelfand spectrum of the C\(^*\)-diagonal \(C\).
We now describe the second component.

In \cite{Raa23} the first author introduced the notion of a labelled Bratteli diagram to graphically encapsulate the information of an AH-algebra arising from generalised diagonal connecting maps. This is a refinement of the classic Bratteli diagram, which in addition to keeping track of the AF-subsystem, also keeps track of the spaces and eigenfunctions between building blocks.

The spectrum $\widehat{C}$ may be described as the infinite path space on the labelled Bratteli diagram with respect to the eigenfunctions which make up the edges of this path, allowing us to compute explicit models for various examples later. The tail equivalence groupoid for labelled Bratteli diagrams is a natural generalisation of the tail equivalence groupoid model of an AF-algebra (see \cite{KumjianPaskRaeburnRenault_GraphsGrpdsCKAlgs}), incorporating the necessary extra topological information into the path structure.

Assume $A$ is a diagonal AH-algebra. 
The \emph{associated labelled Bratteli diagram} $\mathcal{B}_{\lambda}=(V,E,r,s,\lambda)$ is an infinite simple graph with vertices $V$ and edges $E$ such that $V = \bigsqcup\limits_{n=0}^\infty V_n$ and $E = \bigsqcup\limits_{n=0}^\infty E_n$, where $$V_n=\{X_{n,1}, \ldots, X_{n,N_n}\}, \; E_n=\bigsqcup\limits_{  y \in Y(n,(i,j))}\{e_y\}.$$ Here $Y(n,(i,j))$ is a finite indexing set, with cardinality equal to the multiplicity of the embedding $C(X_{n,j})\otimes M_{k_{n,j}} \hookrightarrow C(X_{n+1,i})\otimes M_{k_{{n+1},i}}$. The source and range maps $r,s$ are defined by $s: E_n \to V_n$, $s(e_y)=X_{n,j}$ and $r: E_n \to V_{n+1}$, $r(e_y)=X_{n+1,i}$ whenever $y\in Y(n,(i,j)).$
We assume that $V_0 = \{0\}$ and that $E_0= \bigsqcup\limits_{y=1}^{\sum\limits_{j=1}^{N_1}k_{1,j}}\{0\}.$

We label the edges by using the function 
\begin{equation}
\lambda: E \to \bigcup_{n =1}^\infty \bigcup\limits_{j=1}^{N_{n}} \bigcup\limits_{i=1}^{N_{n+1}} C(X_{n+1,i},X_{n,j})
\end{equation}
such that $\lambda(e_y)=\lambda_y \in C(r(e), s(e))$ is the corresponding eigenfunction associated to edge $e_y$. 

Note that the underlying graph $(V,E,r,s)$ of a diagonal AH-system simplifies to the (unlabelled) Bratteli diagram $\mathcal{B}$ of its AF-subsystem.

Let $A$ be a diagonal AH-algebra and let $\mathcal{B}_{\lambda}=(V,E,r,s,\lambda)$ be the associated labelled Bratteli diagram.
A (labelled) path in $\mathcal{B}_{\lambda}$ is a finite or infinite sequence \(x^\lambda=x^{\lambda}_1\ldots x^\lambda_n\) or $x^{\lambda}=x^{\lambda}_1x^{\lambda}_2\ldots$ where each $x^{\lambda}_i=(e_i,t_i)$ is composed of an edge $e_i$ together with a point $t_i\in s(e_i)$ such that the sequence $e_1e_2e_3\ldots$ forms a path in the Bratteli diagram obtained by restricting to the AF-subsystem and $\lambda(e_i)(t_{i+1})=t_i$ for all $i$. We may write this as $x^{\lambda}=(e_i,t_i)_{i=1}^{\infty}$. We will denote the set of infinite paths in $\mathcal{B}_{\lambda}$ starting at the root of \(\mathcal B_\lambda\) by $\mathcal{P}_{\mathcal{B}_{\lambda}}^{\infty}$.
For a path $x^{\lambda}$, we will write $x^{\lambda}_{[k,m]}$ for the finite portion of the path $x^{\lambda}_k\ldots x^{\lambda}_n$ and we will write $x^{\lambda}_{[n,\infty)}$ for the portion $x^{\lambda}_n x^{\lambda}_{n+1}\ldots$.

Two paths $x^{\lambda}_{[k,m]}=(e_i,t_i)_{i=k}^m$ and $y^{\lambda}_{[m+1,n]}=(e^{\prime}_i,t^{\prime}_i)_{i=m+1}^n$ can be concatenated to form a path $z^{\lambda}_{[k,n]}=x^{\lambda}_{[k,m]}y^{\lambda}_{[m+1,n]}$ when $s(e^{\prime}_{m+1})=r(e_{m})$ and when $\lambda(e_m)(t^{\prime}_{m+1})=t_m$. When $n=\infty$ we form the natural generalisation of the above to infinite paths.

We will also consider the concatenation of an unlabelled finite path $p_{[k,m]}=(e_i)_{i=k}^m$ with a labelled path $y^{\lambda}_{[m+1,n]}=(f_i,t_i)_{i=m+1}^n$ whenever $s(f_{m+1})=r(e_m)$. When this is the case we obtain the labelled path $p_{[k,m]}y^{\lambda}_{[m+1,n]}=(g_i,t_i)_{i=k}^n$ where $g_i=e_i$ for $k\le i\le m$ and $g_i=f_i$ for $m+1\le i\le n$, and where $t_m=\lambda(e_m)(t_{m+1})$, and inductively defining $t_s=\lambda(e_{s})(t_{s+1})$ for $k\le s\le m-1.$ When $n=\infty$ we form the natural generalisation of the above to infinite paths. 

\begin{definition}
\label{AH:def:cylinder-set}
Let $p_{[1,n]}=(e_i)_{i=1}^n$ be a finite (unlabelled) path in the Bratteli diagram $\mathcal{B}$ of the AF-subsystem. Let $U \subseteq s(e_n)$ be an open set (when the edges are viewed inside $\mathcal{B}_{\lambda}$).
The \emph{cylinder set} $Z(p_{[1,n]},U)$ associated to \(p_{[1,n]}\) and \(U\) is
\begin{equation*}
    Z(p_{[1,n]},U):=\{p_{[1,n]}x_{[n+1,\infty)}^\lambda: x^{\lambda}_{[n+1,\infty)}=(e_i,t_i)_{i={n+1}}^\infty, s(e_{n+1})=r(p_n), t_n=\lambda(e_n)(t_{n+1})\in U\},
\end{equation*}
that is, \(Z(p_{[1,n]},U)\) consists of all paths $p_{[1,n]}x^{\lambda}_{[n+1,\infty)}\in \mathcal{P}^\infty_{\mathcal B_\lambda}$ whose $n^{\text{th}}$ coordinate is of the form $(e_n,t)$ where $t\in U.$
\end{definition}

\begin{proposition}
\label{AH:prop:cylinder-sets-basis}
Let $\mathcal{B}_{\lambda}=(V,E,r,s,\lambda)$ be the associated labelled Bratteli diagram for a diagonal AH-system.
The collection of all cylinder sets $Z(p_{[1,n]},U)$ for finite unlabelled paths $p_{[1,n]}=(e_i)_{i=1}^n$ and open sets $U \subseteq s(e_n)$ form a base for a topology on $\mathcal{P}_{\mathcal{B}_{\lambda}}^\infty$.
\end{proposition}
\begin{proof}
    Let $x^{\lambda}=(e_i,t_i)_{i=1}^\infty \in \mathcal{P}_{\mathcal{B}_{\lambda}}^\infty.$ 
    It is clear that $x^{\lambda}\in Z((e_1),0)$, where \(0\in V_0\) is the space serving as the root of the labelled Bratteli diagram, and so each infinite path belongs to some cylinder set. 
    Now assume an infinite path $x^{\lambda}=(e_i,t_i)_{i=1}^{\infty}$ belongs to the intersection $Z(p_{[1,n]},U)\cap Z(q_{[1,m]},V)$ for some path $p_{[1,n]}=(f_i)_{i=1}^n$ with $U$ open in $s(f_n)$ and some path $q_{[1,m]}=(g_i)_{i=1}^m$ with $V$ open in $s(g_m)$. Assume without loss of generality that $n\le m.$ It then follows that $f_i=e_i$ for all $1\le i\le n$ and that $g_i=e_i$ for all $1\le i\le m.$ Let $W:= V \cap (\lambda(e_n) \circ\cdots \circ\lambda(e_{m-1}))^{-1}(U),$ which is open as all the eigenfunctions \(\lambda(e_i)\) are continuous. We then have \(t_n\in W\) and hence $x^{\lambda}\in Z(q_{[1,m]},W)\subset Z(p_{[1,n]},U)\cap Z(q_{[1,m]},V)$ by construction.
\end{proof}

We equip $\mathcal{P}_{\mathcal{B}_{\lambda}}^{\infty}$ with the topology generated by the base of these cylinder sets. 

\begin{definition}\label{def:labelled_Brat_Groupoid}
    In the setting above, we will say that two paths $x^{\lambda}, y^{\lambda} \in \mathcal{P}_{\mathcal{B}_{\lambda}}^{\infty}$ are \emph{tail-equivalent} if $x^{\lambda}_{[n,\infty)}=y^{\lambda}_{[n,\infty)}$ for some $n\in\mathbb{N}$. Define $\mathcal{G}_{\mathcal{B}_{\lambda}}$ as the pair groupoid arising from this tail equivalence relation. We equip $\mathcal{G}_{\mathcal{B}_{\lambda}}$ with the subspace topology inherited from the product topology on $\mathcal{P}_{\mathcal{B}_{\lambda}}^{\infty}\times \mathcal{P}_{\mathcal{B}_{\lambda}}^{\infty}$ (which necessarily makes the inverse and multiplication maps continuous).
\end{definition}

\begin{definition}
    Let $p_{[1,n]}=(e_i)_{i=1}^n$ and $q_{[1,n]}=(f_i)_{i=1}^n$ such that $s(e_n)=s(f_n)$. Let $U$ be an open subset of $s(e_n)$. Then we define a \emph{double cylinder set}  $Z(p_{[1,n]},q_{[1,n]},U)$ as consisting of all $(x^{\lambda},y^{\lambda})\in\mathcal{G}_{\mathcal{B}_{\lambda}}$ such that $x^{\lambda}_{[n,\infty)}=y^{\lambda}_{[n,\infty)}$ and such that $x^{\lambda} \in Z(p_{[1,n]},U)$ and $y^{\lambda}\in Z(q_{[1,n]},U)$.
\end{definition}

\begin{lemma}\label{lem:GBlambda}
    $\mathcal{G}_{\mathcal{B}_{\lambda}}$ is an \'etale groupoid. A basis for its topology is given by double cylinder sets $Z(p_{[1,n]},q_{[1,n]},U).$
\end{lemma}

\begin{proof}
Let us first show that the double cylinder sets form a basis for the topology on $\mathcal{G}_{\mathcal{B}_{\lambda}}.$
Fix a pair \newline$(x^{\lambda},y^{\lambda})=((e_i,t_i)_{i=1}^{\infty},(f_i,s_i)_{i=1}^{\infty})\in \mathcal{G}_{\mathcal{B}_{\lambda}}$ and let \(n\in\Nz\) be large enough so that $x^{\lambda}_{[n,\infty)}=y^{\lambda}_{[n,\infty)}$. Then $$(x^{\lambda},y^{\lambda})\in Z((e_i)_{i=1}^{n},(f_i)_{i=1}^{n},s(e_n)) \subset (Z((e_i)_{i=1}^{n}, s(e_n))\times Z((f_i)_{i=1}^{n}, s(e_n))) \cap  \mathcal{G}_{\mathcal{B}_{\lambda}}$$ and for general finite paths $p_{[1,n]}$ and $q_{[1,m]}$ with $m\ge n$ we have
\begin{equation}\label{eq:basis_double_cyl}(Z(p_{[1,n]}, U)\times Z(q_{[1,m]}, V)) \cap  \mathcal{G}_{\mathcal{B}_{\lambda}}=\underset{k\ge m,\;\; a_{[n,k]}, \;\;b_{[m,k]}}\bigcup Z(p_{[1,n]}a_{[n+1,k]},q_{[1,m]}b_{[m+1,k]}, W_{k,a_{[n+1,k]}, b_{[m+1,k]}})\end{equation} 
where $a_{[n+1,k]}$, $b_{[m+1,k]}$ are any paths such that the concatenations $p_{[1,n]}a_{[n+1,k]},q_{[1,m]}b_{[m+1,k]}$ are finite paths terminating at the same vertex.
The sets \(W_{k,a_{[n+1,k]},b_{[m+1,k]}}\) are defined in the following way: letting $\lambda_a$ denote the (right to left) composition of eigenfunctions labelling the path $a_{[n+1,k]}$ and $\lambda_b$ denote the (right to left) composition of eigenfunctions labelling the path $b_{[m+1,k]}$, we define $W_{k,a_{[n+1,k]}, b_{[m+1,k]}}:=\lambda_a^{-1}(U)\cap \lambda_b^{-1}(V)$, which is necessarily open. Hence we have shown that the double cylinder sets indeed form a basis. 

Let us now show that $\mathcal{G}_{\mathcal{B}_{\lambda}}$ is \'etale. The range map is open and injective on the double cylinder set $Z(p_{[1,n]},q_{[1,n]},U)$, mapping it to the basic open set $Z(p_{[1,n]},U)$ in the unit space. Hence the range map is a local homeomorphism and so $\mathcal{G}_{\mathcal{B}_{\lambda}}$ is \'etale. 
\end{proof}

The next step is to show that $\mathcal{G}_{\mathcal{B}_{\lambda}}$ is a groupoid model for a diagonal AH-algebra described by a diagonal AH-system. 
The aim is to show that $\mathcal{G}_{\mathcal{B}_{\lambda}}$ is isomorphic to the transformation groupoid $\mathcal{G}_{\mathcal B}\ltimes \widehat C$ and invoke Theorem \ref{thm: Af action on spectrum}. In order to do this we first show how we can naturally identify $\widehat{C}$ with $\mathcal{P}_{\mathcal{B}_{\lambda}}^{\infty}.$ 

We recall that by \cite[Theorem~3.6]{BL17} we have $\widehat{C}=\varprojlim(\widehat{C_n},\dot{p_n})$ endowed with the inverse limit topology. 
We can write $\widehat{C_n}=\bigsqcup\limits_{i=1}^{N_n}\widehat{C_{n,i}}$ and by the construction given on page 19 of \cite{Li18} we can in fact treat $\dot{p_n}$ as a disjoint union $\bigsqcup\limits_{i,j, y\in Y(n,(i,j))}p_y$ where by \cite[Section~3~\&~Lemma~3.3]{Raa23} each $p_y$ is in fact one of the corresponding eigenfunctions from $X_{n+1,j}$ to $X_{n,i}$. 
Recall from \eqref{eq:canonCartInBlocks} that the diagonals \(C_n\) are of the form
\[C_n=\bigoplus_{i=1}^{N_n}C_{n,i}, \quad C_{n,i}=C(X_{n,i})\otimes D_{k_{n,i}}.\]
We can identify each $\chi\in\widehat{C}$ with a sequence $(e_n, \chi_n)_{n=1}^{\infty}$ where each $\chi_n$ belongs to $X_{n,i_n}$ and \(e_n\in \widehat D_{k_{n,i_n}}\) for some $i_n\in\{1,2,\ldots,N_n\}$, and where the bonding maps are the eigenfunctions. 
Hence if $e_n$ is the edge in $\mathcal{B}_{\lambda}$ connecting the nodes $X_{n,i_n}$ with $X_{n+1,i_{n+1}}$ the associated eigenfunction $\lambda(e_n)$ will satisfy $\lambda(e_n)(\chi_{n+1})=\chi_n.$ 

\begin{lemma}\label{lem:spectrumC_path_space}
    In the setting above, define $h:\widehat{C}\rightarrow\mathcal{P}_{\mathcal{B}_{\lambda}}^{\infty}$ by $h(\chi)=(e_n,\chi_n)_{n=1}^{\infty}$. Then $h$ is an isomorphism of topological spaces. 
\end{lemma}

\begin{proof}
    That $h(\chi)\in \mathcal{P}_{\mathcal{B}_{\lambda}}^{\infty}$ follows immediately from the construction above. Assume $h(\chi)=(e_n,\chi_n)_{n=1}^{\infty}=h(\chi^{\prime})=(e^{\prime}_n,\chi^{\prime}_n)_{n=1}^{\infty}$. Therefore $(\chi_n)_{n=1}^{\infty}=(\chi^{\prime}_n)_{n=1}^{\infty}$ with the same bonding maps and so they define the same element in the inverse limit. Hence $\chi=\chi^{\prime}$ and $h$ is injective. It is trivial to see that $h$ is surjective. 

    Now we show that $h$ preserves the topology. Let $U\subset X_{n,i}$ be an open set and consider the basic open set $\pi_n^{-1}(U)\subset \widehat{C}$, where $\pi_n$ is the canonical projection from the inverse limit onto the $n^{\text{th}}$ coordinate. Let $\{p_a\}_{a\in A}$ denote the finite collection of all unlabelled paths which start at the node corresponding to $\{0\}$ in $\mathcal{B}_{\lambda}$ and terminate on the node corresponding to $X_{n,i}$ in $\mathcal{B}_{\lambda}$. Then $h$ maps $\pi_n^{-1}(U)$ precisely onto the union $$\bigcup\limits_{a\in A}Z(p_{a}, U)$$ which is open in $\mathcal{P}_{\mathcal{B}_{\lambda}}^{\infty}$.
    
    Conversely, let $Z(p_{[1,n]},U)$ be a basic open set in $\mathcal{P}_{\mathcal{B}_{\lambda}}^{\infty}$. Assume $p_{[1,n]}=e_1e_2\ldots e_n$ where if viewed inside $\mathcal{B}_{\lambda}$ we have $s(e_k)=X_{k,i_k}.$ Then $h^{-1}$ maps $Z(p_{[1,n]},U)$ precisely onto the finite intersection $\left(\bigcap\limits_{k=1}^{n-1}\pi_k^{-1}(X_{k,i_k})\right)\cap \pi_n^{-1}(U)$ which is open in the inverse limit topology.
\end{proof}

We write $H:C(\mathcal{P}_{\mathcal{B}_{\lambda}}^{\infty})\rightarrow C$ for the C$^*$-algebra isomorphism  which is dual to the homeomorphism $h$ in Lemma \ref{lem:spectrumC_path_space}.  

\begin{lemma}\label{lem:dual_maps}
    Let $\rho:\widehat{C}\rightarrow \mathcal{P}_{\mathcal{B}}^{\infty}$ be the anchor map of the action \eqref{eq:AFActionOnDiag} in Proposition \ref{prop:AFGrpdActOnAHSpace}. Define $G:\mathcal{P}_{\mathcal{B}_{\lambda}}^{\infty}\rightarrow\mathcal{P}_{\mathcal{B}}^{\infty}$ as the map sending $x^{\lambda}=(e_i,t_i)_{i=1}^{\infty}$ to $x=(e_i)_{i=1}^{\infty}.$ Then $\rho=G\circ h.$
\end{lemma}

\begin{proof}
    Note that there is a canonical inclusion of C$^*$-algebras $i: C(\mathcal{P}_{\mathcal{B}}^{\infty})\hookrightarrow C(\mathcal{P}_{\mathcal{B}_{\lambda}}^{\infty})$ defined as follows: if $p=e_1e_2\ldots e_n$ and $s(e_n)=X_{n,i}$ (when the edges are viewed in $\mathcal{B}_{\lambda}$) then $i(1_{Z_p})=1_{Z(p,X_{n,i})}$. It is clear that the Gelfand dual to $i$ is the map $G$. 
    
    Let $\chi=(\chi_k)_{k=1}^{\infty}\in \widehat{C}$ such that $h(\chi)=(f_i,\chi_i)_{i=1}^{\infty}$ Then $$H(i(1_{Z_p}))(\chi)=H(1_{Z(p,X_{n,i})})(\chi)=1_{Z(p,X_{n,i})}(h(\chi))=1_{Z(p,X_{n,i})}((f_i,\chi_i)_{i=1}^{\infty})=\delta_{p,(f_i)_{i=1}^n}.$$ Now the canonical inclusion of $C(\mathcal{P}_{\mathcal{B}}^{\infty})$ inside $C$ considered in Proposition \ref{prop:AFGrpdActOnAHSpace} (and whose dual map is $\rho$) satisfies $1_{Z_p}(\chi)\neq 0$ if and only if $\chi_k\in X_{k,s(e_k)}$ for $1\le k\le n$, which in turn is true if and only if $p=(f_i)_{i=1}^n$. 
    Hence this canonical inclusion is equal to $H\circ i$. 
    By uniqueness of the Gelfand dual map, it follows that the dual $\rho$ must equal the dual to $H\circ i$, which is $G\circ h$ and the proof is complete.
\end{proof}

For a labelled path \(x^\lambda\in\mathcal P^\infty_{\mathcal B_\lambda}\) we write \(x:=G(x^\lambda)\in\mathcal P^\lambda_{\mathcal B}\) for the unlabelled component (here \(G\) is the map described by Lemma~\ref{lem:dual_maps}).

\begin{lemma}\label{lem:h_inverse_actioned}
    Let $x^{\lambda}=(e_i,t_i)_{i=1}^{\infty}$ and $y^{\lambda}=(f_i,s_i)_{i=1}^{\infty}$ such that $x^{\lambda}_{[n,\infty)}=y^{\lambda}_{[n,\infty)}$.
    Let $h$ be the homeomorphism in Lemma \ref{lem:spectrumC_path_space}.
    Under the action \eqref{eq:AFActionOnDiag} we have \begin{equation}\label{eq:h_inverse_action}(x,y)\cdot h^{-1}(y^{\lambda})=h^{-1}(x^{\lambda}).\end{equation}
\end{lemma}

\begin{proof}
    First note that the left hand side of \eqref{eq:h_inverse_action} is well defined since by Lemma \ref{lem:dual_maps} it follows that $\rho(h^{-1}(y^{\lambda}))=F(y^{\lambda})=s((F(x),F(y)))$. 
    Hence there exists a unique $w^{\lambda}=(g_i,r_i)_{i=1}^{\infty}$ such that $(x,y)\cdot h^{-1}(y^{\lambda})=h^{-1}(w^{\lambda}).$ We aim to show that $g_i=e_i$ and $r_i=t_i$ for all $i\in\mathbb{N}$ which will imply that $w^{\lambda}=x^{\lambda}.$

    Let $U\subset s(e_n)$ be any open set containing $t_n$. 
    Let $U_k=(\lambda(e_n)\circ \lambda(e_{n+1})\circ\cdots \circ \lambda({e_{k-1}}))^{-1}(U)$ for $k\ge n$. 
    For each $k\ge n$, consider the continuous function $c_k=1_{Z(x_{[1,k]},U_k)}\in C(\mathcal{P}_{\mathcal{B}_{\lambda}}^{\infty})\cong C$. 
    Then it follows that for all $k\ge n$ we have
    \begin{equation*}
        \begin{split}
            (x,y)\cdot h^{-1}(y^{\lambda})(c_k)&= h^{-1}(y^{\lambda})(1_{Z_{x_{[1,n]},y_{[1,n]}}}^*c_k1_{Z_{x_{[1,n]},y_{[1,n]}}}) \\
            &= h^{-1}(y^{\lambda})(1_{Z(y_{[1,k]},U_k)})\\
            &= 1_{Z(y_{[1,k]},U_k)}(y^{\lambda})\\&= 1,
        \end{split}
    \end{equation*}
    where the last equality follows because $s_k=t_k$ for all $k\ge n$. 
    Hence it must follow that $h^{-1}(w^{\lambda})(c_k)=1$ for all $k\ge n$. 
    This means that $1_{Z(x_{[1,k]},U_k)}(w^{ \lambda})=1$ for all $k\ge n$, which means the initial \([1,k]\)-segments of the unlabelled paths \(w\) and \(x\) coincide for all $k\ge n$.
    But this simply means $w=x$. 
    The above computation also shows $r_n\in U$, and since \(U\) is an arbitrary open neighbourhood of \(t_n\), we see $r_n=t_n$.
    This forces that $r_i=t_i$ for all $i\in \mathbb{N}$, and hence $w^{\lambda}=x^{\lambda}$ as desired. 
\end{proof}

\begin{theorem}\label{thm:groupoid_as_labeled_Brat}
    Let \((A_n,\phi_n)_{n\in\Nz}\) be a diagonal AH-system and let $A$ be the resulting AH-algebra.
   
    The function $\Phi:\mathcal{G}_{\mathcal{B}_{\lambda}}\rightarrow \mathcal{G}_{\mathcal B}\ltimes \widehat C$ defined by
    \begin{equation}\label{eq:explicit_iso_brat_to_transformation}\Phi(x^{\lambda},y^{\lambda})=((x,y),h^{-1}(y^{\lambda}))\end{equation}
    is an isomorphism of topological groupoids. In particular $\mathcal{G}_{\mathcal{B}_{\lambda}}$ is a groupoid model for $A$.

    \begin{proof}
   By Lemma \ref{lem:h_inverse_actioned} we have that $\Phi$ is well-defined. Let us show that $\Phi$ is a groupoid homomorphism. Let $x^{\lambda}, y^{\lambda}, z^{\lambda}\in\mathcal{P}_{\mathcal{B}_{\mathcal{\lambda}}}^{\infty}$ such that $x^{\lambda}_{[n,\infty)}=y^{\lambda}_{[n,\infty)}=z^{\lambda}_{[n,\infty)}$. 
   Lemma \ref{lem:h_inverse_actioned} implies that \(\Phi\) preserves composable pairs, and we obtain 
   $$\Phi((x^{\lambda}, y^{\lambda}),(y^{\lambda}, z^{\lambda}))=\Phi((x^{\lambda}, z^{\lambda}))=((x,z),h^{-1}(z^{\lambda}))=((x,y),h^{-1}(y^{\lambda}))((y,z),h^{-1}(z^{\lambda}))=\Phi((x^{\lambda}, y^{\lambda}))\Phi((y^{\lambda}, z^{\lambda})),$$
   so $\Phi$ is a groupoid homomorphism.

    Let us now show that $\Phi$ is bijective.
    Given $(x,y)\in \mathcal{G}_{\mathcal{B}}$ with $x=(e_i)_{i=1}^{\infty}, y=(f_i)_{i=1}^{\infty}$ and $e_i=f_i$ for $i\ge n$, and given $\chi\in \widehat{C}$ such that $\rho(\chi)=y$, it follows from Lemma \ref{lem:dual_maps} that $G\circ h (\chi)=y$ and so $h(\chi)=(f_i,s_i)_{i=1}^{\infty}.$ Define $x^{\lambda}=(e_i,t_i)_{i=1}^{\infty}$ where $t_i=s_i$ for $i\ge n$ and where $t_{n-k}=\lambda(e_{n-k})(t_{n-k+1})$ for $k=1,\ldots n-1$. Let $y^{\lambda}=h(\chi).$ Then $(x^{\lambda},y^{\lambda})\in\mathcal{G}_{\mathcal{B}_{\lambda}}$ and $\Phi((x^{\lambda},y^{\lambda}))=((x,y),\chi)$ and so $\Phi$ is surjective. 
    Since $h$ is a homeomorphism it follows immediately that $\Phi$ is injective. 
    Hence $\Phi$ is bijective. 
    
    It remains to show that $\Phi$ preserves the topological structure. It is clear that $\Phi$ identifies basic open double cylinder sets of the form $Z(p_{[1,n]},q_{[1,n]},U)$ of $\mathcal{G}_{\mathcal{B}_{\lambda}}$ with basic open sets $(Z_{p,q}\times \pi_n^{-1}(U))\cap (\mathcal{G}_{\mathcal{B}}\ltimes \widehat{C})$. Hence $\Phi$ is an isomorphism of topological groupoids. 
    \end{proof}
\end{theorem}

\begin{corollary}
    Let $A$ be a diagonal AH-algebra with Weyl groupoid $\mathcal{G}_A$. 
    Then $$A\cong C^*_{\mathrm r}(\mathcal{G}_A)\cong C^*_{\mathrm r}( \mathcal{G}_{\mathcal B}\ltimes \widehat C)\cong C^*_{\mathrm r}(\mathcal{G}_{\mathcal{B}_{\lambda}}).$$
\end{corollary}

\section{Examples}\label{sec:Examples}

We now use the machinery developed to describe the groupoid models for some examples. These examples are known to have groupoid models (due to having a Cartan subalgebra, see \cite[Corollary~4.4]{LR22}).

\subsection{Villadsen algebras of the first kind (\cite{Vil98})}

For $i\in\mathbb{N}$, select $n_i,k_i\in\mathbb{N}$ and set $A_i=C(X^{n_i})\otimes M_{k_i}$, where $X$ is a compact Hausdorff space. Assume that $k_i|k_{i+1}$ and that the sequence $\{k_i\}$ has no upper bound. We also assume that $n_i|n_{i+1}$. We can write $X^{n_{i+1}}=(X^{n_i})^{m_i}$ where $n_im_i=n_{i+1}.$ The Villadsen algebra of the first kind is $A=\varinjlim(A_i,\phi_i)$, where the connecting maps $\phi_i$ are defined as follows: $$\phi_i(f)=\mathrm{diag}(f\circ \pi^i_1,f\circ\pi^i_2,\ldots,f\circ \pi^i_{p_i},f\circ c^i_1,f\circ c^i_2,\ldots, f\circ c^i_{q_i}),$$  and where $$\pi^i_s: X^{n_{i+1}}\rightarrow X^{n_i}$$ is some canonical projection onto a coordinate component for $s=1,\ldots, p_i$, and where $$c^i_r: X^{n_{i+1}}\rightarrow X^{n_i}$$ is a constant function for $r=1,\ldots, q_i$. We allow $q_i$ to be $0$ (meaning that connecting maps may not have constant eigenfunctions). By construction we have $p_i+q_i=\frac{k_{i+1}}{k_i}$, which is the multiplicity of the embedding, so in particular the connecting maps are unital. The connecting maps are injective since \(p_i\geq 1\). Examples of Villadsen algebras of the first kind include the following:

\begin{itemize}
    \item \textbf{Toms' examples of non-classifiable C$^*$-algebras} \cite{Toms}. Let $X=[0,1]$ and let $n_i=6N_i$ where $N_i=\prod\limits_{j\le i}\overline{n_j}$, and where $\overline{n_1}=1.$ It is also assumed that $k_1=4$. Let $p_i=\overline{n_i}$ and assume $p_i>>q_i$ as $i\to \infty$, and that for each $r\in\mathbb{N}$ there exists $i_0$ such that $r\vert\frac{k_{i_0+1}}{k_{i_0}}$.
    \item \textbf{Goodearl Algebras} \cite{Goo}. Let $n_i=1$ for all $i$ (so $\pi^i_s$ are all identity maps) and assume $q_i\ge 1$. Assume further that the union of the images of all the constant maps $c^i_r$ is dense in $X$ (implying $A$ is simple, see \cite[Example~3.1.7]{Ror}). 
\end{itemize}

The AF-subsystem for $A$ will give rise to a Bratteli diagram $\mathcal{B}=(V,E,r,s)$ where $V_i$ consists of only one element, and where $E_i$ consists of $\frac{k_{i+1}}{k_i}$ elements. The labelled Bratteli diagram $\mathcal{B}_{\lambda}=(V,E,r,s,\lambda)$ will have $V_i=\{X^{n_i}\}$ and $\lambda$ labels the edges on $E_i$ by the appropriate eigenfunctions corresponding to it; either a canonical projection or a constant map. Let $C$ be the associated canonical C$^*$-diagonal of $A$ and let $\widehat{C}$ be its spectrum. 
Let $e=e_1e_2\ldots\in \mathcal{P}^{\infty}_{\mathcal{B}}$ be an infinite path starting at the first vertex. The fibre space $\varprojlim(X^{n_i},\lambda(e_i))$ corresponding to the path \(e\) can be described explicitly depending on various cases.

\underline{Case 1.} The sequence $\{n_i\}_i$ has a strictly increasing subsequence.

\underline{Case 1a.} The sequence $\{\lambda(e_i)\}_{i}$ contains infinitely many constant maps. Then it is clear that $\varprojlim(X^{n_i},\lambda(e_i))\cong \{*\}.$

\underline{Case 1b.} The sequence $\{\lambda(e_i)\}_{i}$ is eventually only projection maps. In this case $\varprojlim(X^{n_i},\lambda(e_i))\cong \prod\limits_{i}X^{n_i}\cong \prod_{\Nz}X =: X^{\infty}$

In Case 1, $\widehat{C}$ is a space fibred over $\mathcal{P}_{\mathcal{B}}^{\infty}$, where each fibre is either $\{*\}$ or $X^{\infty}.$ 

\underline{Case 2.} There exists $i_0$ such that $n_{i+1}=n_i$ for all $i\ge i_0$.

\underline{Case 2a.} The sequence $\{\lambda(e_i)\}_{i}$ contains infinitely many constant maps. Then it is clear that $\varprojlim(X^{n_i},\lambda(e_i))\cong \{*\}.$

\underline{Case 2b.} The sequence $\{\lambda(e_i)\}_{i}$ is eventually only projection maps. It is clear then that these must be the identity and so $\varprojlim(X^{n_i},\lambda(e_i))\cong X^{n_I}$ where $n_I$ is the largest element of $\{n_i\}$ such that $n_I\neq n_{I-1}.$

Hence in Case 2 $\widehat{C}$ is a space fibred over $\mathcal{P}_{\mathcal{B}}^{\infty}$, where each fibre is either $\{*\}$ or $X^{n_I}.$

As a set we consider $\widehat{C}$ as a bundle over \(\mathcal P^\infty_{\mathcal B}\). In order to topologise it we import the inverse limit topology from $\widehat{C}$. By Theorem \ref{thm: Af action on spectrum} we obtain the following:

\begin{proposition}\label{prop: example VillFirst}
    The groupoid model for a general Villadsen algebra of the first kind $A$ is given by $$\mathcal{G}_{\mathcal{B}}\ltimes E_{\mathcal{P}_{\mathcal{B}}^{\infty}}$$ where $E$ is fibred over $\mathcal{P}_{\mathcal{B}}^{\infty}$, the space of infinite paths on the associated Bratteli diagram $\mathcal{B}$. If $X$ is the underlying compact Hausdorff space in the inductive system defining \(A\), then each fibre of \(E\) is of the form $\{*\}, X^{n}$, or $X^{\infty}$ (depending on the Villadsen algebra in question).
\end{proposition}

\subsection{AH-models of dynamical systems}

Let \(X\) be a compact Hausdorff space and let \(\sigma\colon X\to X\) be a homeomorphism.
The \emph{AH-model} for the dynamical system \((X,\sigma)\) is the AH-system
\[C(X)\xrightarrow{\phi_0} C(X, M_2)\xrightarrow{\phi_1}\cdots\to C(X,M_{2^n})\xrightarrow{\phi_n} C(X,M_{2^{n+1}})\to\cdots\]
where the connecting maps are given by
\[\phi_n\colon f\mapsto \begin{pmatrix}
    f&0\\
    0&f\circ\sigma
\end{pmatrix}.\]
By identifying \(C(X,M_k)\) with \(M_k\otimes C(X)\) in the canonical way we see that the connecting maps are indeed diagonal.

Write \(A_n:=C(X,M_{2^n})\) and let \(C_n\subseteq A_n\) be the diagonal in each building block; the algebra of functions \(f\in A_n\) such that \(f(x)\) is diagonal for each \(x\in X\).
Let \(A:=\varinjlim_n A_n\) and \(C:=\varinjlim_n C_n\) be the resulting AHDM-algebra and its canonical diagonal respectively.
Each \(C_n\) has spectrum \(\widehat {C_n}\cong\{0,1\}^n\times X\) and the dual maps \(\hat\phi_n\colon \widehat{C_{n+1}}\to\widehat{C_n}\) are given by
\begin{align*}
    \hat\phi_n((t_1,\dots,t_{n+1}),x)&=\begin{cases}
    ((t_1,\dots,t_n),x),& t_{n+1}=0,\\
    ((t_1,\dots,t_n),\sigma(x)),& t_{n+1}=1,
\end{cases}\\
&=((t_1,\dots,t_n),\sigma^{t_{n+1}}(x))
\end{align*}
for \((t_1,\dots,t_{n+1})\in\{0,1\}^{n+1}\) and \(x\in X\).
Since \(\hat\phi_n((t_1,\dots,t_{n+1}),x)=\hat\phi_n((s_1,\dots,s_{n+1}),y)\) implies \((t_1,\dots,t_n)=(s_1,\dots,s_n)\), a generic element of the projective limit \(\varprojlim_n\{0,1\}^n\times X\) can be expressed in the form
\begin{equation}\label{eq:AHModelDynamSeq}
    (t_{[1,n]},x_n)_n
\end{equation}
where \(t\in\{0,1\}^\Nz\) is a sequence, \(t_{[1,n]}\) is the initial segment of length \(n\) of \(t\), and \((x_n)_n\) is an appropriate sequence in \(x\).
Since \(\sigma\) is a homeomorphism, we see that the sequence \((x_n)_n\) in \eqref{eq:AHModelDynamSeq} is uniquely determined by the sequence \(t\in\{0,1\}^\Nz\) and point \(x_1\in X\) since \(x_n=\sigma^{t_{n+1}}(x_{n+1})\) inductively determines the other values \(x_n\).
Moreover, any choice of starting point \(x_1\in X\) and sequence \(t\in \{0,1\}^\Nz\) gives rise to a unique element of the projective limit in this way, so the map \(b\colon \varprojlim_n\{0,1\}^n\times X\to \{0,1\}^\Nz\times X\) given by \(b((t_1,\dots,t_n),x_n)_n=\left(t,x\right)\) is a bijection (where \(t\in \{0,1\}^\Nz\) is the sequence with initial segments \((t_1,\dots, t_n)\)).
The inverse of \(b\) is given by
\[b^{-1}(t,x)=\left(t_{[1,n]},\sigma^{-\sum_{i=2}^n t_i}(x)\right)_n.\]

\begin{proposition}\label{prop: exampleAHDynam}
    The map \(b\) is a homeomorphism \(\{0,1\}^\Nz\times X\cong\varprojlim_n \{0,1\}^n\times X\).
    The canonical groupoid model for \(A\) is isomorphic to \(G_{2^\infty}\ltimes (\{0,1\}^\Nz\times X)\), where \(G_{2^\infty}\) is the canonical groupoid model for the CAR algebra \(M_{2^\infty}\).
    \begin{proof}
        The topology on \(\varprojlim_n\{0,1\}^n\times X\) is the coarsest topology such that the projections \(\pi_k\colon \varprojlim_n\{0,1\}^n\times X\to\{0,1\}^k\times X\) given by
        \[\pi_k((t_1,\dots,t_n),x_n)_n)=((t_1,\dots,t_k),x_k)\]
        are continuous.
        Precomposing with the bijection \(b\) yields maps
        \[\pi_k^\sigma:=\pi_k\circ b^{-1}\colon (t,x)\mapsto (t_{[1,k]},\sigma^{-\sum_{i=2}^kt_i}(x_1)).\]
        Fix \(p\in\{0,1\}^k\) and \(U\subseteq X\).
        The preimage of \(\{p\}\times U\) under \(\pi_k^\sigma\) consists exactly of pairs \((t,x)\) with \(t_{[1,k]}=p\) and \(\sigma^{-\sum_{i=2}^kp_i}(x)\in U\), so this preimage is given by
        \[Z_p\times \sigma^{\sum_{i=2}^kp_i}(U)\]
        where \(Z_p\subseteq \{0,1\}^\Nz\) is the basic open set of sequences with initial segment \(p\).
        Since \(\sigma\) is a homeomorphism, the set \(Z_p\times \sigma^{\sum_{i=2}^kp_i}(U)\) is open in the product topology.
        Conversely, if \(Z_p\times V\) is a basic open set in the product topology then it is realised as the preimage under \(\pi_k^\sigma\) of \(\{p\}\times\sigma^{-\sum_{i=2}^k p_i}(V)\).
        Hence \(b\) is a homeomorphism and \(\hat D\cong\{0,1\}^\Nz\times X\).
        
        Theorem~\ref{thm: Af action on spectrum} then implies that the groupoid model for \(A\) is isomorphic to the transformation groupoid \(G_{2^\infty}\ltimes (\{0,1\}^\Nz\times X)\), proving the final claim.
    \end{proof}
\end{proposition}

\begin{remark}
    Under the homeomorphism \(b\), the canonical projection \(\pi_k\colon\varprojlim_n\{0,1\}^n\times X\to \{0,1\}^k\times X\) is mapped to the `\(\sigma\)-twisted' projection \(\pi_k^\sigma\colon (t,x)\mapsto (t_{[1,k]},\sigma^{-\sum_{i=2}^k}(x))\) as opposed to the `obvious' projection \((t,x)\mapsto (t_{[1,k]},x)\).
    Hence, while the spaces are homeomorphic, the universal cone under \(\{0,1\}^\Nz\times X\) for the projective system \(\widehat{D_n}\xleftarrow{\hat\phi_n}\widehat{D_{n+1}}\) must incorporate the homeomorphism \(\sigma\) in order to encode the relevant universal property for the system. 
\end{remark}

\end{document}